\documentclass{amsart}
\usepackage[margin=1in]{geometry}
\usepackage{amssymb}
\usepackage{amsmath}

\newtheorem{theorem}{Theorem}[section]
\newtheorem{proposition}[theorem]{Proposition}
\newtheorem{corollary}[theorem]{Corollary}

\theoremstyle{definition}
\newtheorem{definition}[theorem]{Definition}
\newtheorem{example}[theorem]{Example}

\theoremstyle{remark}
\newtheorem{remark}[theorem]{Remark}

\numberwithin{equation}{section}

\makeatletter
\renewcommand{\@biblabel}[1]{[#1]}
\makeatother

\usepackage{hyperref}

\begin{document}

\title{The Wave Equation in the Context of Reduced Group $C^*$-algebras}



\author{Fan Huang}
\address{Fan Huang, School of Mathematical and Statistical Sciences,
Arizona State University, Tempe, Arizona 85287}
\curraddr{}
\email{fhuang21@asu.edu}
\thanks{}


\date{\today}



\begin{abstract}

Motivated by the identification $C(\mathbb{T})\cong C_r^*(\mathbb{Z})$ and the wave equation on the circle, we explore the wave equation in the context of reduced group $C^*$-algebras $C_r^*(G)$ for countably infinite, possibly non-abelian groups $G$. Using a one-parameter group of $*$-automorphisms whose infinitesimal generator paves the way to an analogue of the Laplacian, we establish the existence and uniqueness of solutions to the wave equation within this framework.

\end{abstract}

\maketitle



\section{Introduction}

The heat equation and the wave equation are two of the most famous partial differential equations. In the simplest forms, they can be formulated in the context of $C(\mathbb{T})$, the continuous complex functions on the circle, with solution given by Fourier series (for example, see \cite{dym1972fourier}). It is well-known that if $G$ is any countably infinite abelian group, via the Fourier transform, one has the identification of $C^*$-algebras $C(\widehat{G})\cong C_r^*(G)$ where $\widehat{G}$ denotes the dual group of $G$. Although such identification is no longer in place when $G$ is not abelian. Taking inspiration from  $C(\mathbb{T})\cong C_r^*(\mathbb{Z})$, the authors in \cite{bedos2024heat}, among many other things, investigated the heat equation in the context of $C_r^*(G)$ for countably infinite, not necessary abelian groups $G$, by applying the Fourier theory in this framework developed in \cite{bedos2009twisted}. The current article explores the corresponding wave equation analogue. 

Concerning the classical situation, let $\Delta $ denote the Laplace operator on $C^{2}(\mathbb{T})$, and $f_0,g_0 \in C(\mathbb{T})$. For the associated wave problem, one is looking for a function $u:[0,\infty)\to C(\mathbb{T})$ that satisfies the following conditions:
\begin{enumerate}
    \item[$\bullet$] $u(0)=f_0$,
    \item[$\bullet$] $u(t) \in C^2(\mathbb{T})$ for every $t>0$,
    \item[$\bullet$] $u$ is twice differentiable on $(0,\infty)$ and $u^{\prime\prime}(t)=\Delta(u(t))$ for every $t>0$,
    \item[$\bullet$] $\lim_{t\to 0^+}\|u(t)-f_0\|_{\infty}=0$,
    \item[$\bullet$] $\lim_{t\to 0^+}\|u^{\prime}(t)-g_0\|_{\infty}=0$. 
\end{enumerate}

The functions $f_0,g_0$ are the given initial data, usually representing respectively the initial displacement and initial velocity of a vibrating string with fixed ends (cf. \cite{stein2011fourier}). Some smoothness conditions are usually imposed on these functions so that a unique solution can be found.

It should be mentioned that in the above formulation, $u(t)$ outputs an element in $C(\mathbb{T})$ for each $t>0$, and the derivative $u^{\prime}(t)$ is defined by the usual formula, with the limit being with respect to the uniform norm $\|\cdot\|_{\infty}$ on $C(\mathbb{T})$. In addition, the spatial variable is hidden, it can be displayed by defining $v:[0,\infty)\times \mathbb{T}\to \mathbb{C}$ as $v(t,x)=u(t)(x)$, more precisely, if $u$ is a solution to the above wave problem, and $v(t,x)$ is defined as such, then $$v_{tt}(t,x)=v_{xx}(t,x),\enspace t\in (0,\infty),\enspace x \in \mathbb{T}.$$  

We also mention that even though in the classical situation, the Fourier series solution to such a problem usually involves the Fourier sine series expansion, it will be evident soon that, for us, it is more appropriate to make use of the complex Fourier series with respect to the characters $e_n(\theta)=e^{in\theta},n \in \mathbb{Z},\theta \in [0,2\pi)$, where $\mathbb{T}$ is identified with $[0,2\pi)$ via $e^{i\theta}\longleftrightarrow \theta$. 

Then one can check by formally differentiating, that $$u(t)=\widehat{f_0}(0)+t\cdot \widehat{g_0}(0)+\sum_{n \neq 0}\bigg(\cos(tn)\widehat{f_0}(n)+\frac{\sin(tn)}{n}\widehat{g_0}(n) \bigg)e_n$$is a candidate for a solution, where $\widehat{f}$ denotes the Fourier transform of a function $f \in C(\mathbb{T})$.

In light of $C(\mathbb{T})\cong C_r^*(\mathbb{Z})$, the correspondence $e_n \longleftrightarrow \lambda(n)$ for $n \in \mathbb{Z}$ where $\lambda$ is the (left) regular representation, and the identity $\Delta(e_n)=-d(n)e_n$ with $d:\mathbb{Z}\to [0,\infty)$ given by $d(n)=n^2$, the wave problem in the setting of $C_r^*(G)$ for a countably infinite group $G$ can be formulated as follows.

Let $d:G \to [0,\infty)$ be a nonzero function and let $H_d^{\mathcal{D}}:\mathcal{D}\to C_r^*(G)$ be a (usually unbounded) linear map whose domain $\mathcal{D}$ is a subspace of $C_r^*(G)$ containing the linear span of the $\lambda(g)$'s such that $$H_{d}^{\mathcal{D}}(\lambda(g))=-d(g)\lambda(g),\enspace g\in G.$$We think of any such linear map as playing the role of the Laplacian.

We say that a function $u:[0,\infty)\to C^{*}_r(G)$ is a solution of the wave problem on $C_r^*(G)$ associated to $H_{d}^{\mathcal{D}}$ and $x_0,y_0 \in C_r^*(G)$, if $u$ satisfies the following conditions:

\begin{itemize}
    \item[$\bullet$] $u(0)=x_0$,
    \item[$\bullet$] $u(t)\in \mathcal{D}$ for every $t>0$,
    \item[$\bullet$] $u$ is twice differentiable on $(0,\infty)$ and $u^{\prime\prime}(t)=H_{d}^{\mathcal{D}}(u(t))$ for every $t>0$,
    \item[$\bullet$] $\lim_{t\to 0^+}\|u(t)-x_0\|=0$,
    \item[$\bullet$] $\lim_{t\to 0^+}\|u^{\prime}(t)-y_0\|=0.$
\end{itemize}

We mention that such a formulation depends on the prescribed function $d:G \to [0,\infty)$ and a linear map $H_d^{\mathcal{D}}$, including its domain $\mathcal{D}$, with the above effect. We are primarily interested in such a problem where the domain $\mathcal{D}$ is strictly larger than the linear span of the $\lambda(g)$'s. The main difficulty comes from the justification of the validity of the solution when the problem is formulated with an arbitrary operator $H_d^{\mathcal{D}}$. One way to deal with this is to restrict the function $d$ and potentially the operator $H_d^{\mathcal{D}}$ to certain more tractable kinds. For example, the case of the heat problem on $C_r^*(G)$ relies on $d:G \to [0,\infty)$ being negative definite (see section 4 of \cite{bedos2024heat}). The wave version is different.

We will show that for each nonzero homomorphism $b:G \to (\mathbb{R},+)$, letting $d:=b^2$, there is an operator $H_{d}^{\mathcal{D}}$ defined on a reasonably large domain, for which the above wave problem on $C_r^*(G)$ associated to $H_{d}^{\mathcal{D}}$ and some $x_0,y_0 \in C_r^*(G)$ has a unique solution, given by applying the appropriate operators to the initial data, provided that they belong to some suitable subspaces.  And in this case, the solution $u(t)$, which now outputs an element in $C_r^*(G)$ for each $t\geq 0$, has a convergent Fourier series in $C_r^*(G)$ for each $t>0$, with the Fourier series expansion given by $$u(t)=\sum_{g \in \text{ker}(b)}\big(\widehat{x_0}(g)+t\widehat{y_0}(g)\big)\lambda(g)+\sum_{g \not \in \text{ker}(b)}\bigg( \cos(tb(g))\widehat{x_0}(g)+\frac{\sin(tb(g))}{b(g)}\widehat{y_0}(g)\bigg)\lambda(g).$$

It should be mentioned here that while the uniform convergence of Fourier series in $C(\mathbb{T})$ is usually understood in the conditional sense, due to the absence of a natural order on an arbitrary group $G$, it is more appropriate for us, given how the Fourier theory in $C_r^*(G)$ is developed, to interpret convergence in the unconditional sense.

A nonzero homomorphism of a group $G$ into the additive real numbers is sometimes called a real character of $G$. For specific countably infinite groups $G$, one can usually write down an explicit formula for such a function, and in general, such a function exists if and only if $G/[G,G]$ is torsion-free. In particular, this is the case if $G$ is torsion-free abelian (see for example 24.35 of \cite{hewitt1997abstract}).

This paper is organized as follows. Section 2 contains some preliminaries about unordered summation of series in Banach spaces, the reduced group $C^*$-algebra, and its the Fourier theory. Section 3 contains the set up for the wave equation in the abstract context and section 4 contains the material on the existence and uniqueness of the wave problem, making use of the tools developed in the preceding section. The final section includes some further questions and comments.

\vspace{.1in}

\noindent\textbf{Acknowledgements.} The author wishes to thank Erik B\'edos and Steven Kaliszewski for many helpful comments and suggestions regarding this article.

\section{Preliminaries}

The notation and terminology will largely be the same as in \cite{bedos2024heat}, but for the ease of the reader, we recall the major ones that will be needed here. Throughout this paper, we assume that $G$ is a countably infinite group with identity $e$.

\subsection{Unordered Summation}

Let $S$ be a nonempty set and $\{x_s\}_{s \in S}$ be a family of elements in a normed space $(X,\|\cdot\|)$. We say that the (formal) series $\sum_{s \in S}x_s$ converges to $x \in S$ and write $\sum_{s \in S}x_s=x$, if the net of finite partial sums $\big\{\sum_{s\in F}x_s\big\}_{F \in \mathcal{F}}$ converges to $x$, where $\mathcal{F}$ denotes the set of all finite subsets of $S$, directed by inclusion: more precisely, for any $\epsilon >0$, there exists $F_{\epsilon} \in \mathcal{F}$ such that $\|x-\sum_{s \in F}x_s\|<\epsilon$ for every $F \in \mathcal{F}$ containing $F_{\epsilon}$. For more information, see \cite{bedos2024heat} and \cite{Huang2025Wave}

We will frequently use the following facts concerning unordered sums.

\begin{proposition}\label{2.1}
Assume $S$ is a countably infinite set, $X$ is a Banach space, and $\{x_s\}_{s \in S}$ is a family of elements in $X$.
\begin{enumerate}
    \item[$(i)$] $\sum_{s \in S}x_s$ is convergent if and only if $\sum_{s \in S}\varphi(s)x_s$ is convergent for every $\varphi \in \ell^{\infty}(G)$.
    \item[$(ii)$] If $S = A\cup B$ with $A\cap B=\emptyset$, then $\sum_{s \in S}x_s$ is convergent if and only if both $\sum_{s \in A}x_s$ and $\sum_{s \in B}x_s$ are convergent, in which case $$\sum_{s \in S}x_s = \sum_{s \in A}x_s +\sum_{s \in B}x_s.$$ 
    \item[$(iii)$] If $X = \mathbb{R}$ or $X = \mathbb{C}$, then $$\sum_{s \in S}x_s \text{ is convergent if and only if }\sum_{s \in S}|x_s| \text{ is convergent.}$$
\end{enumerate}
\end{proposition}

\subsection{Fourier Theory in $C_r^*(G)$}

Let $\ell^2(G)$ denote the Hilbert space of square-summable functions on $G$. The left regular representation $\lambda$ of $G$ on $\ell^2(G)$ is defined by $$(\lambda(g)\xi)(h)=\xi(g^{-1}h),\enspace \xi \in \ell^2(G),\enspace g,h\in G.$$

Let $B(\ell^2(G))$ denote the space of all bounded linear operators from $\ell^2(G)$ into itself and $\|\cdot\|$ denote the operator norm on $B(\ell^2(G))$. Let $\mathbb{C}(G):=\text{Span}(\lambda(G))$, called the group algebra. 

The reduced group $C^*$-algebra $C_r^*(G)$ (resp. the group von Neumann algebra $\text{vN}(G)$) is the $C^*$-subalgebra (resp. von Neumann subalgebra) of $B(\ell^2(G))$ generated by $\lambda(G)$. That is $$C_r^*(G):=\overline{\mathbb{C}(G)}^{\|\cdot\|}\enspace \text{ and }\enspace \text{vN}(G):=\overline{\mathbb{C}(G)}^{\text{WOT}}$$ where WOT refers to the weak operator topology.

Let $\{\delta_h\}_{h \in G}$ denote the canonical basis of $\ell^2(G)$, we have $$\lambda(g)\delta_h=\delta_{gh},\enspace g ,h\in G.$$In particular, $\lambda(g)\delta_e=\delta_g$ for all $g \in G$.

The canonical (faithful normal) tracial state $\tau$ on $\text{vN}(G)$ is given by $\tau(x)=\langle x\delta_e,\delta_e\rangle$, and the associated norm on $\text{vN}(G)$ is defined by $\|x\|_{\tau}=\tau(x^*x)^{1/2}$. The map $$x\mapsto \widehat{x}:=x\delta_e$$is an injective linear map from $\text{vN}(G)$ into $\ell^2(G)$, called the Fourier transform of $x$. It satisfies $\tau(x)=\widehat{x}(e)$ and $\|x\|_{\tau}=\|\widehat{x}\|_2\leq \|x\|$, moreover, $$\widehat{x}(g)=\tau(x\lambda(g)^*)\enspace\text{ for all }g \in G.$$

The Fourier series of $x \in \text{vN}(G)$ is the formal series $\sum_{g \in G}\widehat{x}(g)\lambda(g)$. It always converges to $x$ w.r.t. $\|\cdot\|_{\tau}$, but it is not necessarily convergent w.r.t. operator norm for $x \in C_r^*(G)$. Set $$CF(G):=\bigg\{ x \in C_r^*(G):\sum_{g \in G}\widehat{x}(g)\lambda(g) \text{ is convergent in operator norm}\bigg\}.$$If $x \in CF(G)$, then the Fourier series of $x$ necessarily converges to $x$ (see proposition 2.12 in \cite{bedos2009twisted}).

Concerning the Fourier coefficients, the following proposition will be useful, it follows readily from proposition 2.10 in \cite{bedos2009twisted}.

\begin{proposition}\label{2.2}

Assume $\xi:G \to \mathbb{C}$ is such that $\sum_{g \in G}\xi(g)\lambda(g)$ is convergent with respect to operator norm, set $$y:=\sum_{g \in G}\xi(g)\lambda(g) \in C_r^*(G).$$Then $\widehat{y}=\xi$ and $y \in CF(G)$.

\end{proposition}

One important tool that plays a crucial role is that of multipliers.

Let $\varphi \in \ell^{\infty}(G)$. We recall that $\varphi$ is called a multiplier on $G$, i.e., $\varphi \in MA(G)$, if and only if there exists a (necessarily unique) bounded linear operator $M_{\varphi}:C_r^*(G) \to C_r^*(G)$ such that $$M_{\varphi}(\lambda(g))=\varphi(g)\lambda(g) \enspace \text{ for all }g \in G.$$For $x \in C_r^*(G)$ and $g \in G$, we have that $\widehat{M_{\varphi}(x)}(g)=\varphi(g)\widehat{x}(g)$, so the Fourier series of $M_{\varphi}(x)$ is given by $$\sum_{g \in G}\varphi(g)\widehat{x}(g)\lambda(g).$$If in addition $M_{\varphi}(x)\in CF(G)$, then $$M_{\varphi}(x)=\sum_{g \in G}\varphi(g)\widehat{x}(g)\lambda(g).$$Concerning $MA(G)$, it is easy to see that $\ell^2(G)\subseteq MA(G)$. Any positive definite function on $G$ is also in $MA(G)$ (see lemma 1.1 in \cite{haagerup1978example}).

\section{Machinery}

From now on assume that $b:G \to (\mathbb{R},+)$ is a nonzero homomorphism (for example, see \ref{4.5}). By Theorem 1 (and its proof) in \cite{Lance1976UnboundedDO}, there exists a one-parameter group $\{M_t^b\}_{t \in \mathbb{R}}$ of $*$-automorphisms of $C_r^*(G)$ satisfying

\begin{itemize}
    \item[$\bullet$] $M_0^b$ is the identity operator on $C_r^*(G)$, $M_{s+t}^b=M_s^b\circ M_t^b$ for all $s,t \in \mathbb{R}$,
    \item[$\bullet$] the map $t\mapsto M_t^b(x)$ from $\mathbb{R}$ to $C_r^*(G)$ is norm-continuous for every $x \in C_r^*(G)$ (w.r.t. operator norm).
    \item[$\bullet$] $M_t^b(\lambda(g))=e^{tib(g)}\lambda(g)$ for all $g \in G, t \in \mathbb{R}$.
\end{itemize}

The third item says that in particular, for each $t \in \mathbb{R}$, the map $g\mapsto e^{tib(g)}$ belongs to $MA(G)$ and hence for $x \in C_r^*(G)$, we have $$\widehat{M_t^b(x)}(g)=e^{tib(g)}\widehat{x}(g),\enspace g \in G.$$

The general theory of one-parameter $C_0$-semigroups adapted to the case of one-parameter groups (see for example, \cite{applebaum2019semigroups}, \cite{bratteli2012operator} and p.79 in \cite{engel2000one}) gives that the set

$$\mathcal{C}=\bigg\{x \in C_r^*(G):\lim_{t\to 0}\frac{M_t^b(x)-x}{t}\text{ exists in }C_r^*(G) \bigg\}$$is a dense subspace of $C_r^*(G)$. The infinitesimal generator $\delta_b^{\mathcal{C}}:\mathcal{C}\to C_r^*(G)$ of $\{M_t^b\}_{t\in \mathbb{R}}$, which is defined by $$\delta_b^{\mathcal{C}}(x)=\lim_{t\to 0}\frac{M_t^b(x)-x}{t}\text{ for every }x \in \mathcal{C},$$is a closed linear operator. Moreover, we have
\begin{itemize}
    \item[$\bullet$] $M_t^d(\mathcal{C}) \subseteq \mathcal{C}$ and $M_t^b(\delta_b^{\mathcal{C}}(x))=\delta_b^{\mathcal{C}}(M_t^b(x))$ for all $t\in \mathbb{R}$ and all $x \in \mathcal{C}$,
    \item[$\bullet$] for each $x_0 \in \mathcal{C}$, the map $u:[0,\infty)\to C_r^*(G)$ given by $u(t):=M_t^b(x_0)$ is the (unique) solution of the inital value problem $$u^{\prime}(t)=\delta_b^{\mathcal{C}}(u(t))\text{ for }t>0,\enspace u(0)=x_0,$$ 
    \item[$\bullet$] for each $x_0 \in \mathcal{C}$, the map $v:[0,\infty)\to C_r^*(G)$ given by $v(t)= M_{-t}^b(x_0)$ is the (unique) solution of the initial value problem $$v^{\prime}(t)=-\delta_b^{\mathcal{C}}(v(t)) \text{ for }t>0,\enspace v(0)=x_0.$$
\end{itemize}

An alternative description of $\mathcal{C}$ is that it consists of those $x \in C_r^*(G)$ for which there exists a $y \in C_r^*(G)$ such that $$\lim_{t\to 0}\psi\bigg(\frac{M_t^b(x)-x}{t} \bigg)=\psi(y)$$for all $\psi \in C_r^*(G)^*$, in which case we have $\delta_b^{\mathcal{C}}(x)=y$.

One can easily check that $\mathbb{C}(G)\subseteq \mathcal{C}$ and that $$\delta_b^{\mathcal{C}}(\lambda(g))=ib(g)\lambda(g),\enspace g\in G.$$

We view $\delta_b^{\mathcal{C}}$ as an unbounded operator playing the role of the ordinary differential operator in the current setting. Composing it twice we get an operator that's an analogue of the Laplacian, except that we have to make sure that the composition is well-defined in this context. It is well-known that the degree of smoothness of a function $f \in C(\mathbb{T})$ is reflected in the rate of decay of the Fourier coefficients $\widehat{f}(n)$. The following definition corresponds to the latter in the current context.

\begin{definition}\label{3.1}
    Let $n \in \mathbb{N}$, define $F^n(G)$ to be the set of elements $x \in C_r^*(G)$ for which the following two conditions hold:
    \begin{itemize}
        \item[$\bullet$] $$\sum_{g \in G}b(g)^k\widehat{x}(g)\lambda(g)\text{ is convergent w.r.t. $\|\cdot\|$ for }k=0,1,2,...,n,$$
        \item[$\bullet$] $$\sum_{g \not \in \text{ker}(b)}b(g)^{-k}\widehat{x}(g)\lambda(g)\text{ is convergent w.r.t. $\|\cdot\|$ for }k=0,1,2,...,n.$$ 
    \end{itemize}
\end{definition}

One can check that for each $n \in \mathbb{N}$, $F^n(G)$ is a (dense) subspace of $C_r^*(G)$ containing $\mathbb{C}(G)$ and $F^{n+1}(G)\subseteq F^n(G)$. In addition, $F^n(G) \subseteq CF(G)$. We think of $F^n(G)$ as the analogue of $C^n(\mathbb{T})$ in this context.

\begin{proposition}\label{3.2}
    For $n \in \mathbb{N}$ and $t \in \mathbb{R}$, $M_t^b(F^n(G)) \subseteq F^n(G)$.
\end{proposition}

\begin{proof}
    Let $n \in \mathbb{N}, t\in \mathbb{R}$, and $x \in F^n(G)$. To see that $M_t^b(x) \in  F^n(G)$, let $k \in \{0,1,2...,n\}$. We need to show that the series $$\sum_{g \in G}b(g)^k\widehat{M_t^b(x)}(g)\lambda(g)\enspace \text{ and }\enspace\sum_{g \not \in \text{ker}(b)}b(g)^{-k}\widehat{M_t^b(x)}(g)\lambda(g)$$are convergent. We know that $\widehat{M_t^b(x)}(g)=e^{tib(g)}\widehat{x}(g)$ for all $g \in G$, so the two series above are respectively $$\sum_{g \in G}e^{tib(g)}b(g)^k\widehat{x}(g)\lambda(g)\enspace \text{ and }\enspace\sum_{g \not \in \text{ker}(b)}e^{tib(g)}b(g)^{-k}\widehat{x}(g)\lambda(g),$$which are both convergent by (i) of Proposition \ref{2.1} and the assumption that $x \in F^n(G)$. 
\end{proof}

\begin{proposition}\label{3.3}
    $F^1(G) \subseteq \mathcal{C}$ and for $x \in F^1(G)$, $$\delta_b^{\mathcal{C}}(x)=\sum_{g \in G}ib(g)\widehat{x}(g)\lambda(g),$$with the series being convergent w.r.t. $\|\cdot\|$.

\end{proposition}

\begin{proof}
    Let $x \in F^1(G)$. Set $$y:=\sum_{g \in G}ib(g)\widehat{x}(g)\lambda(g) \in C_r^*(G).$$We need to show that $x \in \mathcal{C}$ and $\delta_b^{\mathcal{C}}(x)=y$. To this end, it suffices to show that for all $\psi\in C_r^*(G)^*$,
    
    \begin{align}
        \lim_{t \to 0}\psi\bigg(\frac{M_t^b(x)-x}{t} \bigg)=\psi(y).\label{equation 3.1}
    \end{align}
    Let $\psi \in C_r^*(G)^*$.

    Since $x \in F^1(G)$, we know that $x$ and $M_t^b(x)$ both belong to $CF(G)$, and we may replace them by their (convergent) Fourier series expansion. Using the assumption that $\psi$ is bounded, to prove (\ref{equation 3.1}), it suffices to show that 

    \begin{align*}
        \lim_{t\to 0}\sum_{g \in G}\bigg(\frac{e^{tib(g)}-1}{t}-ib(g) \bigg)\widehat{x}(g)\psi(\lambda(g))&=0
    \end{align*}

    which is equivalent to

    \begin{align}
        \lim_{t\to 0}\sum_{g \in G}\eta(tb(g))b(g)\widehat{x}(g)\psi(\lambda(g))=0\label{equation 3.2}
    \end{align}

    where $\eta:\mathbb{R}\to \mathbb{C}$ is the function given by 

    \[\eta(x)= \begin{cases} 
      \frac{e^{ix}-1}{x}-i & \text{ if }x\neq 0 \\
      0 & \text{ if }x=0 
   \end{cases}
\]Note that $\eta$ is uniformly bounded (for example, by $2$) and continuous.

Let $(t_n)_{n=1}^{\infty}$ be a sequence in $\mathbb{R}$ converging to $0$ and consider the function $$f_n(g)= \eta(t_nb(g))b(g)\widehat{x}(g)\psi(\lambda(g)).$$ For each $g \in G$, we have $\lim_{n\to \infty}f_n(g)=0$. In addition, $|f_n(g)|\leq 2|b(g)\widehat{x}(g)\psi(\lambda(g))|$ for all $n\in \mathbb{N},g \in G$. Since $x \in F^1(G)$, by $(iii)$ of Proposition \ref{2.1}, the function $g \mapsto 2|b(g)\widehat{x}(g)\psi(\lambda(g))|$ belongs to $\ell^1(G)$. Therefore by the dominated convergence theorem, $$\lim_{n\to \infty}\sum_{g \in G}f_n(g)=0,$$which proves (\ref{equation 3.2}). It follows that $F^1(G)\subseteq \mathcal{C}$.

\end{proof}

\begin{corollary}\label{3.4}
    For $x \in F^1(G)$, $$\widehat{\delta_b^{\mathcal{C}}(x)}(g)=ib(g)\widehat{x}(g),\enspace g \in G,$$and $\delta_b^{\mathcal{C}}(x) \in CF(G)$.
\end{corollary}

\begin{proof}
    The first part follows by applying Proposition \ref{2.2} to the result of Proposition \ref{3.3}. From the first part, for $x \in F^1(G)$, the Fourier series of $\delta_b^{\mathcal{C}}(x)$ is $$\sum_{g \in G}\widehat{\delta_b^{\mathcal{C}}(x)}(g)\lambda(g)=\sum_{g \in G}ib(g)\widehat{x}(g)\lambda(g),$$which is convergent w.r.t. $\|\cdot\|$,  i.e., $\delta_b^{\mathcal{C}}(x) \in CF(G)$.
\end{proof}

This corollary may be thought of as corresponding (in the context of $C_r^*(G)$) to the identity $\widehat{f^{\prime}}(m)=im\widehat{f}(m),m\in \mathbb{Z}$, for $f \in C^1(\mathbb{T})$. 

\begin{proposition}\label{3.5}
    For $n \in \mathbb{N}$, $\delta_b^{\mathcal{C}}(F^{n+1}(G)) \subseteq F^n(G)$.
\end{proposition}

\begin{proof}
    Let $ n\in \mathbb{N}$ and $x \in F^{n+1}(G)$. By Proposition \ref{3.3}, we know that $\delta_b^{\mathcal{C}}(x) \in C_r^*(G)$. To prove that $\delta_b^{\mathcal{C}}(x) \in F^n(G)$, let $k\in \{0,1,2,...,n\}$, we need to show that the series

\begin{align}\label{expression 3.3}
    \sum_{g \in G}b(g)^k\widehat{\delta_b^{\mathcal{C}}(x)}(g)\lambda(g)\enspace \text{ and }\sum_{g \not \in \text{ker}(b)}b(g)^{-k}\widehat{\delta_b^{\mathcal{C}}(x)}(g)\lambda(g)
\end{align}
are both convergent w.r.t. $\|\cdot\|$. By Corollary \ref{3.4}, these series are respectively

\begin{align}\label{expression 3.3}
    \sum_{g \in G}ib(g)^{k+1}\widehat{x}(g)\lambda(g)\enspace \text{ and }\sum_{g \not \in \text{ker}(b)}ib(g)^{-(k-1)}\widehat{x}(g)\lambda(g)
\end{align}
which are both convergent w.r.t. $\|\cdot\|$ since $x \in F^{n+1}(G)$.

\end{proof}

We now make the following definition. We first define $d:G \to [0,\infty)$ by $d(g)=b(g)^2$.

\begin{definition}\label{3.6}
    Set $\mathcal{D}:=F^2(G)$. We define $H_d^{\mathcal{D}}:\mathcal{D} \to C_r^*(G)$ to be the linear map given by $$H_d^{\mathcal{D}}:=\delta_b^{\mathcal{C}}\circ \delta_b^{\mathcal{C}}.$$
\end{definition}

Note that it is due to the preceding proposition that the above composition is well-defined. From this definition we have $$H_{d}^{\mathcal{D}}(\lambda(g))=-d(g)\lambda(g)$$ for all $g \in G$. We will show later the existence and uniqueness of solution to the wave problem associated to this operator $H_d^{\mathcal{D}}$ and some (given) $x_0\in F^2(G), y_0 \in F^3(G)$.

We will need the following operator, it modifies the Fourier coefficients of an element by $b(g)$ for all $g \not \in \text{ker}(b)$.  

\begin{definition}\label{3.7}
    We define $T_b:F^1(G) \to C_r^*(G)$ by $$T_b(x)=\sum_{g \not \in \text{ker}(b)}b(g)^{-1}\widehat{x}(g)\lambda(g) \enspace (\text{convergence w.r.t. }\|\cdot\|)$$for all $x \in F^1(G)$.
\end{definition}

It follows from Proposition \ref{2.2} that for $x \in F^1(G)$, the Fourier coefficients for $T_b(x)$ are

    \[\widehat{T_b(x)}(g)= \begin{cases} 
      b(g)^{-1}\widehat{x}(g) & \text{ if }g\not \in \text{ker}(b) \\
      0 & \text{ if }g \in \text{ker}(b) 
   \end{cases}
\]and $T_b(x)\in CF(G)$.

\begin{proposition}\label{3.8}
    For $n \in \mathbb{N}$, the following statements are true.
    \begin{enumerate}
        \item[$(i)$] $T_b(F^{n+1}(G)) \subseteq F^{n}(G)$, and
 $$\delta_b^{\mathcal{C}}(T_b(x))=T_b(\delta_b^{\mathcal{C}}(x))=\sum_{g \not \in \text{ker}(b)}i\widehat{x}(g)\lambda(g),\enspace x \in F^{n+1}(G)$$
 \item[$(ii)$] For all $t \in \mathbb{R} $, $$T_b(M_t^b(x))=M_t^b(T_b(x)),\enspace x \in F^n(G)$$
    \end{enumerate}

\end{proposition}

\begin{proof}

Let $n \in \mathbb{N}$.

    $(i)$ Let $x \in F^{n+1}(G)$. To prove that $T_b(x) \in F^n(G)$, let $k \in \{0,1,2,...,n\}$. We need to show that the series

    \begin{align*}
        \sum_{g \in G}b(g)^k\widehat{T_b(x)}(g)\lambda(g)=\sum_{g \not \in \text{ker}(b)}b(g)^{k-1}\widehat{x}(g)\lambda(g)
    \end{align*}
    and 

    \begin{align*}
        \sum_{g \not \in \text{ker}(b)}b(g)^{-k}\widehat{T_b(x)}(g)\lambda(g)=\sum_{g \not \in \text{ker}(b)}b(g)^{-(k+1)}\widehat{x}(g)\lambda(g)
    \end{align*}
    are both convergent w.r.t. $\|\cdot\|$. This follows from the assumption that $x \in F^{n+1}(G)$. This proves $T_b(F^{n+1}(G)) \subseteq F^n(G)$.

Next, let $x \in F^{n+1}(G)$, then the above argument also makes $\delta_b^{\mathcal{C}}(T_b(x))$ a well-defined element in $C_r^*(G)$. From Proposition \ref{3.5}, $T_b(\delta_b^{\mathcal{C}}(x))$ is also a well-defined element in $C_r^*(G)$.

For the equalities in $(i)$, we compute (all series being convergent w.r.t. $\|\cdot\|$):

\begin{align*}
    \delta_b^{\mathcal{C}}(T_b(x))&=\sum_{g \in G}ib(g)\widehat{T_b(x)}(g)\lambda(g)=\sum_{g \not \in \text{ker}(b)}ib(g)b(g)^{-1}\widehat{x}(g)\lambda(g)=\sum_{g \not \in \text{ker}(b)}i\widehat{x}(g)\lambda(g).
\end{align*}

Similarly, using Corollary \ref{3.4}, we get:

\begin{align*}
    T_b(\delta_b^{\mathcal{C}}(x))&=\sum_{g \not \in \text{ker}(b)}b(g)^{-1}\widehat{H_b^{\mathcal{C}}(x)}(g)\lambda(g)=\sum_{g \not \in \text{ker}(d)}b(g)^{-1}ib(g)\widehat{x}(g)\lambda(g)=\sum_{g \not \in \text{ker}(b)}i\widehat{x}(g)\lambda(g).
\end{align*}

\vspace{.1in}

$(ii)$ Let $t \in \mathbb{R}$ and $x \in F^n(G)$. By Proposition \ref{3.2}, $M_t^b(x) \in F^n(G)$ so $T_b(M_t^b(x))$ is well-defined. Since $\widehat{M_t^b(x)}(g)=e^{tib(g)}\widehat{x}(g)$ for all $g \in G$, by the definition of the map $T_b$, we have $$T_b(M_t^b(x))=\sum_{g \not \in \text{ker}(b)}b(g)^{-1}\widehat{M_t^b(x)}(g)\lambda(g)=\sum_{g \not \in \text{ker}(b)}b(g)^{-1}e^{tib(g)}\widehat{x}(g)\lambda(g).$$Next, since $$T_b(x)=\sum_{g \not \in \text{ker}(b)}b(g)^{-1}\widehat{x}(g)\lambda(g),$$with the series being convergent w.r.t. $\|\cdot\|$ and the operator $M_t^b$ being bounded, we get
\begin{align*}
    M_t^b(T_b(x))&=M_t^b\bigg(\sum_{g \not \in \text{ker}(b)}(b(g))^{-1}\widehat{x}(g)\lambda(g) \bigg)\\&=\sum_{g \not \in \text{ker}(b)}M_t^b\bigg((b(g))^{-1}\widehat{x}(g)\lambda(g)\bigg)\\&=\sum_{g \not \in \text{ker}(b)}(b(g))^{-1}e^{tib(g)}\widehat{x}(g)\lambda(g)\\&=T_b(M_t^b(x)).
\end{align*}
\end{proof}

Finally, we will also need the following map as a tool.

\begin{definition}\label{3.9}
    For $x \in CF(G)$, define $S_b:CF(G)\to C_r^*(G)$ by$$S_b(x)=\sum_{g \in \text{ker}(b)}\widehat{x}(g)\lambda(g)\enspace \big(\text{convergence w.r.t. }\|\cdot\|\big)$$
\end{definition}

By $(ii)$ of Proposition \ref{2.1} and the linearity of the map $x\mapsto \widehat{x}$, $S_b$ is a well-defined linear map. By Proposition \ref{2.2}, the Fourier coefficients of $S_b(x)$ for $x \in CF(G)$ are

\[\widehat{S_b(x)}(g)= \begin{cases} 
      \widehat{x}(g) & \text{ if }g \in \text{ker}(b) \\
      0 & \text{ if }g \not \in \text{ker}(b) 
   \end{cases}
\]and $S_b(x) \in CF(G)$.

\begin{proposition}\label{3.10}
    Let $n \in \mathbb{N}$. Then $S_b(F^n(G)) \subseteq F^n(G)$. Moreover, the following statements are true:
    \begin{enumerate}
        \item[$(i)$] $$\delta_b^{\mathcal{C}}(S_b(x))=S_b(\delta_b^{\mathcal{C}}(x))=0,\enspace x\in F^1(G).$$
        \item[$(ii)$] For $t\in \mathbb{R}$, $$M_t^b(S_b(x))=S_b(M_t^b(x))=S_b(x),\enspace x \in CF(G).$$
        \item[$(iii)$]  $$T_b(S_b(x))=S_b(T_b(x))=0,\enspace x \in F^1(G).$$
    \end{enumerate}
\end{proposition}

\begin{proof}
    Let $n \in \mathbb{N}$ and $x \in F^n(G)$. To prove that $S_b(x) \in F^n(G)$, let $k \in \{0,1,2,...,n\}$. Then $$\sum_{g \in G}b(g)^{k}\widehat{S_b(x)}(g)\lambda(g)=\sum_{g \in \text{ker}(b)}b(g)^{k}\widehat{x}(g)\lambda(g)$$which is convergent w.r.t. $\|\cdot\|$ since $x \in F^n(G)$. In addition, $$\sum_{g \not \in \text{ker}(b)}b(g)^{-k}\widehat{S_b(x)}(g)\lambda(g)=0.$$Hence, $S_b(x) \in F^n(G)$.

\vspace{.1in}

$(i)$ Let $x \in F^1(G)$. Then $S_b(x) \in F^1(G)$. Using Proposition \ref{3.3} we get

\begin{align*}
    \delta_b^{\mathcal{C}}(S_b(x))&=\sum_{g \in G}ib(g)\widehat{S_b(x)}(g)\lambda(g)=\sum_{g \in \text{ker}(b)}ib(g)\widehat{x}(g)\lambda(g)=0.
\end{align*}

Moreover, $\delta_b^{\mathcal{C}}(x) \in CF(G)$, so $S_b(\delta_b^{\mathcal{C}}(x))$ is defined, and using Corollary \ref{3.4}, we get

\begin{align*}
    S_b(\delta_b^{\mathcal{C}}(x))&=\sum_{g \in \text{ker}(b)}\widehat{\delta_b^{\mathcal{C}}(x)}(g)\lambda(g)=\sum_{g \in \text{ker}(b)}ib(g)\widehat{x}(g)\lambda(g)=0.
\end{align*}

$(ii)$ Let $t \in \mathbb{R}$ and $x \in F^n(G)$. Since $S_b(x)=\sum_{g \in \text{ker}(b)}\widehat{x}(g)\lambda(g)$ with the series being convergent w.r.t. $\|\cdot\|$, using the boundedness of the operator $M_t^b$, we get
\begin{align*}
    M_t^b(S_b(x))&=M_t^b\bigg(\sum_{g \in \text{ker}(b)}\widehat{x}(g)\lambda(g) \bigg)\\&=\sum_{g \in \text{ker}(b)}M_t^b\bigg(\widehat{x}(g)\lambda(g)\bigg)\\&=\sum_{g \in \text{ker}(b)}e^{tib(g)}\widehat{x}(g)\lambda(g)\\&=\sum_{g \in \text{ker}(b)}\widehat{x}(g)\lambda(g)\\&=S_b(x)
\end{align*}

Finally, using the formula for the Fourier coefficients of $M_t^b(x)$, we get

\begin{align*}
    S_b(M_t^b(x))&=\sum_{g \in \text{ker}(b)}\widehat{M_t^b(x)}(g)\lambda(g)=\sum_{g \in \text{ker}(b)}e^{tib(g)}\widehat{x}(g)\lambda(g)=\sum_{g \in \text{ker}(b)}\widehat{x}(g)\lambda(g)=S_b(x).
\end{align*}

$(iii)$ This follows from the definitions of $T_b(x),S_b(x)$ and the formulas for their Fourier coefficients.

\end{proof}

\begin{theorem}
    Let $t\geq 0$. Then the maps $$\delta_b^{\mathcal{C}},M_t^b,T_b,\text{ and }S_b$$ are all well-defined on $F^{2}(G)$ and any pair of them commute with each other on $F^{2}(G)$.
\end{theorem}

\begin{proof}
    This follows from the previous results.
\end{proof}

\section{The Wave Equation}

We have now developed all the necessary machinery for the formulation and justification of the wave problem on $C_r^*(G)$.

We recall that $b:G \to (\mathbb{R},+)$ is a nonzero homorphism, $d:=b^2$, and $H_d^{\mathcal{D}}:=\delta_b^{\mathcal{C}}\circ \delta_b^{\mathcal{C}}$ is defined on $\mathcal{D}:=F^2(G)$.

\begin{theorem}\label{4.1}
    Let $x_0 \in F^2(G)$, $y_0 \in F^3(G)$. For $t\geq 0$, set $$u_1(t):=t\cdot S_b(y_0),\enspace u_2(t):= \frac{M_t^b(x_0)+M_{-t}^b(x_0)}{2},\enspace u_3(t):=\frac{M_t^b(T_b(y_0))-M_{-t}^b(T_b(y_0))}{2i},\text{ and }$$
 $$u(t):=u_1(t)+u_2(t)+u_3(t).$$Then $u(t)$ is a solution to the wave problem on $C_r^*(G)$ associated to $H_{d}^{\mathcal{D}}$ and $x_0 ,y_0 $.
\end{theorem}    
    
\begin{remark}   
In the classical situation, the initial velocity is usually assumed to be $C^2$. Here the initial velocity is represented by $y_0$ which is assumed to be in $F^3(G)$. This in short, is due to the fact that in the classical case, the term $1/n$ that appears in the Fourier series for $\widehat{g_0}$, helps convergence, whereas here, the corresponding term is $1/b(g)$ which does not help convergence when appearing in a series, as $b$ is only assumed to be a nonzero homomorphism.

\end{remark}

\begin{proof}

For the proof of the theorem, we need to show that this $u:[0,\infty) \to C_r^*(G)$ satisfies the following conditions:

\begin{enumerate}
    \item $u(0)=x_0$,\label{item 1}
    \item $u(t) \in \mathcal{D}$ for all $t>0$,\label{item 2}
    \item $u$ is twice differentiable on $(0,\infty)$ with $u^{\prime\prime}(t)=H_{d}^{\mathcal{D}}(u(t))$, $t>0$.\label{item 3}
    \item $\lim_{t\to 0^+}\|u(t)-x_0\|=0$,\label{item 4}
    \item $\lim_{t\to 0^+}\|u^{\prime}(t)-y_0\|=0$,\label{item 5}
\end{enumerate}



Given that $x_0 \in F^2(G)$ and $y_0 \in F^3(G)$, it is then clear that each expression in the definition of $u(t)$ is well-defined for each $t\geq 0$.

Now we check each of the conditions separately.

\begin{enumerate}
    \item It is clear from the direct substitution of $t=0$ that the function $u(t)$ satisfies condition (\ref{item 1}).

    \item For $t>0$, since $y_0 \in F^3(G) \subseteq F^2(G)$, which is invariant under the linear map $S_b$, we immediately have that $u_1(t) \in \mathcal{D}$. Similarly, because $x_0$ is assumed to be in $F^2(G)$, which is invariant under the operator $M_t^b$ for all $t\in \mathbb{R}$, from the definition, it follows that $u_2(t) \in \mathcal{D}$. The same reasoning together with proposition \ref{3.8} shows that $u_3(t)\in \mathcal{D}$. So condition (\ref{item 2}) is satisfied.
    \item First, it is clear that $u_1(t)$ is twice differentiable with respect to $t$ with $u_1^{\prime\prime}(t)=0$. From Proposition \ref{3.10} we get that

\begin{align*}
    H_{d}^{\mathcal{D}}(u_1(t))&=\delta_{b}^{\mathcal{C}}(\delta_b^{\mathcal{C}}(tS_b(y_0)))=t\delta_b^{\mathcal{C}}(\delta_b^{\mathcal{C}}(S_b(y_0)))=0
\end{align*}

So $u_1^{\prime\prime}(t)=H_{d}^{\mathcal{D}}(u_1(t))$.

    Next, recall that for each $x \in \mathcal{C}$, the maps $t\mapsto M_t^b(x)$ and $t\mapsto M_{-t}^b(x)$ are differentiable on $(0,\infty)$ and satisfy $$\frac{d}{dt}M_t^b(x)=\delta_b^{\mathcal{C}}(M_t^b(x))\enspace \text{ and }\enspace \frac{d}{dt}M_{-t}^b(x)=-\delta_b^{\mathcal{C}}(M_{-t}^b(x)).$$

    Let $t>0$. As $x_0$ belongs to $ F^2(G)$, which is contained in $\mathcal{C}$, we can compute: 

\begin{align*}
    u_2^{\prime}(t)&=\frac{d}{dt}\bigg( \frac{M_t^b(x_0)+M_{-t}^b(x_0)}{2}\bigg)\\&=\frac{\frac{d}{dt}M_t^b(x_0)+\frac{d}{dt}M_{-t}^b(x_0)}{2}\\&=\frac{\delta_b^{\mathcal{C}}(M_t^b(x_0))-\delta_b^{\mathcal{C}}(M_{-t}^b(x_0))}{2}\\&=\frac{M_t^b(\delta_b^{\mathcal{C}}(x_0))-M_{-t}^b(\delta_b^{\mathcal{C}}(x_0))}{2}.
\end{align*}

Further, $u_2^{\prime}(t)$ is differentiable with respect to $t$ since $\delta_b^{\mathcal{C}}(x_0) \in \mathcal{C}$, and we get

\begin{align*}
    u_2^{\prime\prime}(t)&=\frac{d}{dt}u_2^{\prime}(t)\\&=\frac{\frac{d}{dt}M_t^b(\delta_b^{\mathcal{C}}(x_0))-\frac{d}{dt}M_{-t}^b(\delta_b^{\mathcal{C}}(x_0))}{2}\\&=\frac{\delta_b^{\mathcal{C}}(M_t^b(\delta_b^{\mathcal{C}}(x_0)))+\delta_b^{\mathcal{C}}(M_{-t}^b(\delta_b^{\mathcal{C}}(x_0)))}{2}\\&=\delta_b^{\mathcal{C}}\bigg(\delta_{b}^{\mathcal{C}}\bigg(\frac{M_t^b(x_0)+M_{-t}^b(x_0)}{2} \bigg)\bigg)\\&=H_{d}^{\mathcal{D}}(u_2(t)).
\end{align*}

For $u_3(t)$, since $y_0 \in F^3(G)$, $T_d(y_0) \in F^2(G) \subseteq \mathcal{C}$, we can compute similarly:

\begin{align*}
    u_3^{\prime}(t)&= \frac{d}{dt}\bigg(\frac{M_t^b(T_b(y_0))-M_{-t}^b(T_b(y_0))}{2i} \bigg)\\&=\frac{\frac{d}{dt}M_t^b(T_b(y_0))-\frac{d}{dt}M_{-t}^b(T_b(y_0))}{2i}\\&=\frac{\delta_b^{\mathcal{C}}(M_t^b(T_b(y_0)))+\delta_{b}^{\mathcal{C}}(M_{-t}^b(T_b(y_0)))}{2i}\\&=\frac{M_t^b(\delta_b^{\mathcal{C}}(T_b(y_0)))+M_{-t}^b(\delta_b^{\mathcal{C}}(T_b(y_0)))}{2i}
\end{align*}

Since $\delta_b^{\mathcal{C}}(T_b(y_0)) \in \mathcal{C}$, $u_3^{\prime}(t)$ is differentiable with respect to $t$ and 

\begin{align*}
    u_3^{\prime\prime}(t)&=\frac{\frac{d}{dt}M_t^b(\delta_b^{\mathcal{C}}(T_b(y_0)))+\frac{d}{dt}M_{-t}^b(\delta_b^{\mathcal{C}}(T_b(y_0)))}{2i}\\&=\frac{\delta_b^{\mathcal{C}}(M_t^b(\delta_b^{\mathcal{C}}(T_b(y_0))))-\delta_b^{\mathcal{C}}(M_{-t}^b(\delta_b^{\mathcal{C}}(T_b(y_0))))}{2i}\\&=\delta_b^{\mathcal{C}}\bigg(\delta_b^{\mathcal{C}}\bigg( \frac{M_t^b(T_b(y_0))-M_{-t}^b(T_b(y_0))}{2i}\bigg) \bigg)\\&=H_d^{\mathcal{D}}(u_3(t)).
\end{align*}

Combining these equalities we have shown that $u(t)$ satisfies condition (\ref{item 3}).

    \item 

    This follows directly from the continuity property of the one-parameter group $\{M_t^b\}_{t \in \mathbb{R}}$.

\item Using what we have computed above, it is clear that $$\lim_{t\to 0^+}u_1^{\prime}(t)=S_d(y_0)=\sum_{g \in \text{ker}(d)}\widehat{y_0}(g)\lambda(g)\enspace \text{ and }\lim_{t\to 0^+}u_2^{\prime}(t)=0.$$In addition,

\begin{align*}
    \lim_{t\to 0^+}u_3(t)&=\frac{\delta_b^{\mathcal{C}}(T_b(y_0))+\delta_b^{\mathcal{C}}(T_b(y_0))}{2i}=\sum_{g \not \in \text{ker}(b)}\widehat{y_0}(g)\lambda(g)
\end{align*}
Putting these together:

\begin{align*}
    \lim_{t\to 0^+}u^{\prime}(t)&=\sum_{g \in \text{ker}(b)}\widehat{y_0}(g)\lambda(g)+0+\sum_{g \not \in \text{ker}(b)}\widehat{y_0}(g)\lambda(g)=\sum_{g \in G}\widehat{y_0}(g)\lambda(g)=y_0
\end{align*}
This finishes proving item \ref{item 5} and the proof of the theorem.

\end{enumerate}

\end{proof}

For the uniqueness of the solution, the following result is an intermediate step. It picks out the Fourier coefficient of $u(t)$ for each $t\geq 0$ where $u(t)$ is as in Theorem \ref{4.1}. 

\begin{proposition}\label{3.4.3}
    Let $x_0,y_0$ and $u(t)$ be as in Theorem \ref{4.1}. 
    
    If $g  \in \text{ker}(b)$, then $$\widehat{u(t)}(g)=\widehat{x_0}(g)+t\cdot \widehat{y_0}(g)$$

    If $g \not \in \text{ker}(b)$, then

    \begin{align*}
        \widehat{u(t)}(g)&=\cos(tb(g))\widehat{x_0}(g)+\frac{\sin(tb(g))}{b(g)}\widehat{y_0}(g)
    \end{align*}

\end{proposition}

\begin{proof}
  
Recall that

\begin{align*}
    u(t)&=\underbrace{t\cdot S_b(y_0)}_{u_1(t)}+\underbrace{\frac{M_t^b(x_0)+M_{-t}^b(x_0)}{2}}_{u_2(t)}+\underbrace{\frac{M_t^b(T_b(y_0))-M_{-t}^b(T_b(y_0))}{2i}}_{u_3(t)}.
\end{align*}

Let $t>0$ and $g \in \text{ker}(b)$.

First it follows directly from the comment after Definition \ref{3.9} that $\widehat{u_1(t)}(g)=t\cdot \widehat{y_0}(g)$.

Next, using the formulas for the Fourier coefficients of $M_t^b(x_0)$ and $M_{-t}^b(x_0)$, the linearity of the map $x\mapsto \widehat{x}(g)$ and keeping in mind that $g \in \text{ker}(b)$, we have $$\widehat{u_2(t)}=\bigg(\frac{M_t^b(x_0)+M_{-t}^b(x_0)}{2} \bigg)^{\widehat{ \enspace}}(g)=\frac{e^{tib(g)}\widehat{x_0}(g)+e^{-tib(g)}\widehat{x_0}(g)}{2}=\widehat{x_0}(g).$$

From the comments following Definition \ref{3.7}, since $g \in \text{ker}(b)$, $M_t^b(T_b(y_0))^{\widehat{\enspace}}(g)=e^{tib(g)}\widehat{T_b(y_0)}(g)=0$, and $M_{-t}^b(T_b(y_0))^{\widehat{\enspace}}(g)=e^{tib(g)}\widehat{T_b(y_0)}(g)=0$, so $\widehat{u_3(t)}(g)=0$.

Hence, when $g \in \text{ker}(b)$, $\widehat{u(t)}(g)=\widehat{x_0}(g)+t\cdot \widehat{y_0}(g)$.

Now suppose $g \not \in \text{ker}(b)$. Using the same reasoning, it is clear that $\widehat{u_1(t)}(g)=0$ and that $$\widehat{u_2(t)}(g)=\frac{e^{tib(g)}\widehat{x_0}(g)+e^{-tib(g)}\widehat{x_0}(g)}{2}=\bigg(\frac{e^{tib(g)}+e^{-tib(g)}}{2} \bigg)\widehat{x_0}(g)=\cos(tb(g))\widehat{x_0}(g).$$Finally

\begin{align*}
    \widehat{u_3(t)}(g)&=\bigg(\frac{M_t^b(T_b(y_0))-M_{-t}^b(T_b(y_0))}{2i} \bigg)^{\widehat{\enspace}}(g)\\&=\frac{e^{tib(g)}\widehat{T_b(y_0)}(g)-e^{-tib(g)}\widehat{T_b(y_0)}(g)}{2i}\\&=\bigg(\frac{e^{tib(g)}-e^{-tib(g)}}{2i}\bigg)(b(g))^{-1}\widehat{y_0}(g)\\&=\frac{\sin(tb(g))}{b(g)}\widehat{y_0}(g)
\end{align*}

Putting these equalities together completes the proof.

\end{proof}

Using this proposition, we can take a look at the form of the solution when expanded in Fourier series. Let $u(t)$ be given as in Theorem \ref{4.1}. For $t>0$, since $u(t) \in \mathcal{D}=F^2(G)$, it belongs to $CF(G)$, so

\begin{align*}
    u(t)&=\sum_{g \in G}\widehat{u(t)}(g)\lambda(g)\\&=\sum_{g \in \text{ker}(b)}\widehat{u(t)}(g)\lambda(g)+\sum_{g \not \in \text{ker}(b)}\widehat{u(t)}(g)\lambda(g)\\&=\sum_{g \in \text{ker}(b)}(\widehat{x_0}(g)+t\widehat{y_0}(g))\lambda(g)+\sum_{g \not \in \text{ker}(b)}\bigg(\cos(tb(g))\widehat{x_0}(g)+\frac{\sin(tb(g))}{b(g)}\widehat{y_0}(g) \bigg)\lambda(g)
\end{align*}

Letting $G = \mathbb{Z}$ and $b:\mathbb{Z} \to (\mathbb{R},+)$ be the function $b(n)=n$, one recovers the solution to the wave problem on $C_r^*(\mathbb{Z})$ and hence on $C(\mathbb{T})$.

The following is the uniqueness result.

\begin{proposition}

 Suppose $x_0 \in F^2(G),y_0 \in F^3(G)$, and $v:[0,\infty) \to C_r^*(G)$ is a solution of the wave problem associated to the operator $H_d^{\mathcal{D}}$ and $x_0,y_0$. Then $v(t)=u(t)$ for all $t\geq 0$ where $u(t)$ is as defined in Theorem \ref{4.1}.

\end{proposition}

\begin{proof}

The assumption on $v$ says:

\begin{enumerate}
    \item[$\bullet$] $v(0)=x_0$,
    \item[$\bullet$] $v(t) \in \mathcal{D}$ for every $t>0$,
    \item[$\bullet$] $v$ is twice differentiable on $(0,\infty)$ and $v^{\prime\prime}(t)=H_{d}^{\mathcal{D}}(v(t))$, $t>0$,
    \item[$\bullet$] $\lim_{t\to 0^+}v(t)=x_0$,
    \item[$\bullet$] $\lim_{t\to 0^+}v^{\prime}(t)=y_0$.
\end{enumerate}

It is clear that $v(0)=u(0)=x_0$.

Let $t>0$. Because $v(t) \in \mathcal{D}$, it is in the domain of $\delta_b^{\mathcal{C}}$, so Corollary \ref{3.4} gives that for all $g  \in G$, $$\widehat{\delta_b^{\mathcal{C}}(v(t))}(g)=ib(g)\widehat{v(t)}(g)$$In addition, $v(t) \in \mathcal{D}=F^2(G)$ means that $\delta_b^{\mathcal{C}}(v(t)) \in F^1(G)$ which again belongs to the domain of $\delta_b^{\mathcal{C}}$. This not only says that $\delta_b^{\mathcal{C}}(\delta_b^{\mathcal{C}}(v(t))))=:H_d^{\mathcal{D}}(v(t))$ is defined, but also that it is equal to $$\sum_{g \in G}(ib(g))\widehat{\delta_b^{\mathcal{C}}(v(t))}(g)\lambda(g).$$Putting these together and making use of the third item above, we get

\begin{align*}
    v^{\prime\prime}(t)&=H_{d}^{\mathcal{D}}(v(t))=\sum_{g \in G}(ib(g))\widehat{\delta_b^{\mathcal{C}}(v(t))}(g)\lambda(g)=\sum_{g \in G}(ib(g))(ib(g))\widehat{v(t)}(g)\lambda(g)=\sum_{g \in G} -(b(g))^2\widehat{v(t)}(g)\lambda(g)
\end{align*}
with the series on the right hand side being convergent w.r.t. operator norm. Then Proposition \ref{2.2}  gives $$\widehat{v^{\prime\prime}(t)}(g)=-(b(g))^2\widehat{v(t)}(g) \text{ for all }g \in G \text{ and }t >0.$$

For fixed $g \in G$, set $w_g(t):=\widehat{v(t)}(g)$ for all $t>0$. As in the proof of Proposition 4.5 in \cite{bedos2024heat}, using the assumptiont that $v(t)$ is twice differentiable and hence differentiable, we get 

\begin{align*}
    w_g^{\prime}(t)&=\lim_{h\to 0}\frac{w_g(t+h)-w_g(h)}{h}\\&=\lim_{h\to 0}\frac{\widehat{v(t+h)}(g)-\widehat{v(t)}(g)}{h}\\&=\lim_{h\to 0}\bigg(\frac{v(t+h)-v(t)}{h} \bigg)^{\widehat{\enspace}}(g)\\&=\lim_{h\to 0}\tau\bigg(\bigg(\frac{v(t+h)-v(t)}{h}\bigg)\lambda(g)^* \bigg)\\&=\tau\bigg( \bigg(\lim_{h\to 0}\frac{v(t+h)-v(t)}{h}\bigg)\lambda(g)^*\bigg)\\&=\tau(v^{\prime}(t)\lambda(g)^*)\\&=\widehat{v^{\prime}(t)}(g)
\end{align*}

So $$w_g^{\prime}(t)=\widehat{v^{\prime}(t)}(g)\enspace \text{ for all }g \in G,t>0.$$Using that $v^{\prime}(t)$ is differentiable, we get

\begin{align*}
    w_g^{\prime\prime}(t)&=\widehat{v^{\prime\prime}(t)}(g)=-b(g)^2\widehat{v(t)}(g)=-b(g)^2w_g(t)
\end{align*}
for all $g \in G, t >0.$

So for each fixed $g \in G$, we have the following differential equation for $t>0$ $$w_g^{\prime\prime}(t)+b(g)^2w_g(t)=0 \enspace \enspace (*)$$

If $g \in \text{ker}(b)$, $(*)$ becomes $$w_g^{\prime\prime}(t)=0,$$ whose solution is $w_g(t)=a_1t+a_2$ for some constant $a_1,a_2$. Then since $\lim_{t\to 0^+}v(t)=x_0$, we have $$\widehat{x_0}(g)=\lim_{t\to 0^+}\widehat{v(t)}(g)=\lim_{t\to 0^+}w_g(t)=\lim_{t\to 0^+}a_1t+a_2=a_2$$whence $a_2=\widehat{x_0}(g)$.

Next, from the above we have $w_g^{\prime}(t)=a_1$. Now using the assumption that $\lim_{t\to 0^+}v^{\prime}(t)=y_0$ we get 

\begin{align*}
    \widehat{y_0}(g)=\lim_{t\to 0^+}\widehat{v^{\prime}(t)}(g)=\lim_{t\to 0^+}w_g^{\prime}(t)=\lim_{t\to 0^+}a_1=a_1
\end{align*}
whence $a_1 = \widehat{y_0}(g)$.

Thus, we can conclude that if $g \in \text{ker}(b)$, then $$\widehat{v(t)}(g)=w_g(t)=a_2+a_1t=\widehat{x_0}(g)+t\cdot \widehat{y_0}(g)$$

Assuming now that $g\not \in \text{ker}(b)$. Then the solution to $(*)$ is $$w_g(t)=C_g^1e^{tib(g)}+C_g^2e^{-tib(g)}\enspace (t>0)$$for some constants $C_g^1,C_g^2 \in \mathbb{C}$. We then have $$w_g^{\prime}(t)=(C_g^1\cdot i\cdot b(g))e^{tib(g)}+(C_g^2\cdot (-i)\cdot b(g))e^{-tib(g)}$$

Since $\lim_{t\to 0^+}w_g(t)=\lim_{t\to 0^+}\widehat{v(t)}(g)=\widehat{x_0}(g)$, we get

\begin{align*}
    \widehat{x_0}(g)&=\lim_{t\to 0^+}w_g(t)=\lim_{t\to 0^+}C_g^1e^{tib(g)}+C_g^2e^{-tib(g)}=C_g^1+C_g^2.
\end{align*}

Further, since
$ \lim_{t\to 0^+}w_g^{\prime}(t)=\lim_{t\to 0^+}\widehat{v^{\prime}(t)}(g)=\widehat{y_0}(g)$, we also get

\begin{align*}
    \widehat{y_0}(g)&=\lim_{t\to 0^+}w_g^{\prime}(t)=\lim_{t\to 0^+}(C_g^1\cdot i\cdot b(g))e^{tib(g)}+(C_g^2\cdot (-i)\cdot b(g))e^{-tib(g)}=(C_g^1\cdot i\cdot b(g)) + (C_g^2\cdot (-i)\cdot b(g))
\end{align*}

Thus, we get the system

\begin{align*}    C_g^1+C_g^2&=\widehat{x_0}(g),\\
C_g^1-C_g^2&=(-i)\cdot \frac{\widehat{y_0}(g)}{b(g)}. 
\end{align*}

Solving this system gives $$C_g^1=\frac{1}{2}\bigg( \widehat{x_0}(g)-i\cdot \frac{\widehat{y_0}(g)}{b(g)}\bigg)\enspace\text{ and }\enspace C_g^2=\frac{1}{2}\bigg( \widehat{x_0}(g)+i\cdot \frac{\widehat{y_0}(g)}{b(g)}\bigg).$$

Therefore,

\begin{align*}
    \widehat{v(t)}(g)&=w_g(t)\\&=C^1_ge^{tib(g)}+C^2_ge^{-tib(g)}\\&=\frac{1}{2}\bigg( \widehat{x_0}(g)-i\cdot \frac{\widehat{y_0}(g)}{b(g)}\bigg)e^{tib(g)}+\frac{1}{2}\bigg( \widehat{x_0}(g)+i\cdot \frac{\widehat{y_0}(g)}{b(g)}\bigg)e^{-tib(g)}\\&=\widehat{x_0}(g)\bigg( \frac{e^{tib(g)}+e^{-tib(g)}}{2}\bigg)+\widehat{y_0}(g)\cdot \bigg(\frac{1}{b(g)}\bigg)\bigg(\frac{e^{tib(g)}-e^{-tib(g)}}{2i} \bigg)\\&=\cos(tb(g))\widehat{x_0}(g)+\frac{\sin(tb(g))}{b(g)}\widehat{y_0}(g).
\end{align*}

In conclusion, for $t>0$, we have shown that $\widehat{v(t)}(g)=\widehat{u(t)}(g)$ for all $g \in G$. This implies that for all $t>0$, $v(t)=u(t)$, and completes the proof.

\end{proof}

For completeness, we mention a few concrete examples to illustrate the above framework.

\begin{example}\label{4.5} The following is a short list of countably infinite groups $G$ together with a nonzero homomorphism $b:G \to (\mathbb{R},+)$.
    \begin{enumerate}
        \item[$(i)$] Let $G  = \bigoplus_{i=1}^{\infty}\mathbb{Z}$ consists of elements with all but finitely many nonzero entries, denote a typical element by $(a_i)$, define $$b((a_i))=\sum_{i=1}^{\infty}\frac{a_i}{2^i}.$$
        \item[$(ii)$] Let $G$ be the Heisenberg group over $\mathbb{Z}$, let a typical element of $G$ be denoted by  $$[l,m,n]:=\begin{pmatrix}
1 & l & m\\
0 & 1 & n\\
0 & 0 & 1\\
\end{pmatrix},\enspace l,m,n\in \mathbb{Z}. $$Define $b:G \to \mathbb{R}$ by $b([l,m,n])=l$.
\item[$(iii)$] Let $G = \mathbb{F}_2$ be the free group on two generators $x,y$ and let $b(x)=1$ and $b(y)=0$, then extend $b$ linearly to all of $\mathbb{F}_2$, one can check that $b$ is a nonzero homomorphism into the additive real numbers.
    \end{enumerate}
\end{example}

\vspace{.1in}

\section{Further Questions and Comments}

\subsection{}Much of the work done in \cite{bedos2024heat} is to investigate the the so-called heat properties of groups which sheds light on the connection between the algebraic properties of groups and functions defined on them (e.g., amenability, negative definiteness of the function $d:G \to [0,\infty)$) and the solution of the associated heat problem on $C_r^*(G)$. If one is looking for the wave equation analogue of such group properties, the motivation is likely coming from the oscillatory behavior involved in the solution of the wave equation, paralleling how the heat equation is linked to diffusion processes. However, due to the author's limited background in physics, no such properties are proposed at this stage.

\subsection{}Concerning the wave problem on $C_r^*(G)$, there are several directions for generalization. First, as mentioned in the introduction, one may search for more general kind of functions $d:G \to [0,\infty)$ with which the associated wave problem has a justifiable solution. Second, for countably infinite abelian groups $G$, by introducing a 2-cocycle $\sigma:G\times G\to \mathbb{T}$ as done in \cite{bedos2024heat}, one may get the framework for the wave equation in the context of $C_r^*(G,\sigma)$. Furthermore, in view of the work done in \cite{bedos2015fourier}, it may be possible to extend the current work to the setting of discrete $C^*$-cross products $A \rtimes_{\alpha} G$, in which case the wave equation would incorporate the dynamics from an action $\alpha:G \to \text{Aut}(A) $.

\subsection{}Finally, one could investigate the stability of the formulation and solution of the wave problem on groups. For example, given a group $G$ and a subgroup $H$, what are the relationships between the formulation and solution of the wave equation on $C_r^*(G)$ and $C_r^*(H)$, or, the same question may be asked regarding $C_r^*(G_1\times G_2)$ and $C_r^*(G_1)$, $ C_r^*(G_2)$, or that of $C_r^*(G)$ and $C_r^*(G/H)$ when $H$ is normal.

\vspace{.1in}

\bibliographystyle{plain}
\bibliography{reference}

\vspace{.1in}

Address of the author:

\end{document}